\documentclass[11pt]{amsart}
\usepackage[margin=1.5in, centering]{geometry}                
\usepackage{graphicx}
\usepackage{amsfonts,amsthm,amsmath,amssymb,latexsym, amscd, euscript}
\usepackage{graphicx}
\usepackage[dvipsnames]{xcolor}
\usepackage[all]{xy}
\usepackage{epsfig}
\usepackage[parfill]{parskip}
\usepackage{hyperref}
\usepackage{amssymb}
\usepackage{epstopdf}
\usepackage{enumerate}

\newcommand{\br}{\mathbb R}
\newcommand{\bn}{\mathbb N}

\newcommand{\ep}{\epsilon}
\newcommand{\al}{\alpha}

\newcommand{\ssm}{\smallsetminus}
\newcommand{\co}{\colon\thinspace}

\DeclareMathOperator{\arcsinh}{arcsinh}

\theoremstyle{plain}

\newtheorem*{ThmA}{Theorem~\ref*{thm:regular}}
\newtheorem*{ThmD}{Theorem~\ref*{thm:qch}}
\newtheorem*{ThmE}{Theorem~\ref*{thm:arcs}}

\newtheorem{thm}{Theorem}[section]
\newtheorem{cor}[thm]{Corollary}

\newtheorem{prop}[thm]{Proposition}
\newtheorem{lem}[thm]{Lemma}

\newtheorem{rem}[thm]{Remark}

\newtheorem{question}[thm]{Question}

\title{There are no exotic ladder surfaces}
\author{Ara Basmajian}
\address{(A. Basmajian) Department of Mathematics\\
The Graduate Center, City University of New York\\
365 Fifth Ave., New York, NY 10016 and
}
\address{
Department of Mathematics\\
Hunter College, City University of New York\\
695 Park Ave., New York, NY, 10065
}

\author{Nicholas G. Vlamis}
\address{(N.G. Vlamis) Department of Mathematics\\
Queens College, City University of New York\\
65-30 Kissena Blvd., Flushing, NY 11367, U.S.A.
}

\begin{document}

\begin{abstract}
It is an open problem to provide a characterization of quasiconformally homogeneous Riemann surfaces.
We show that given the current literature, this problem can be broken into four open cases with respect to the topology of the underlying surface.
The main result is a characterization in one of the these open cases; in particular, we prove that every quasiconformally homogeneous ladder surface is quasiconformally equivalent to a regular cover of a closed surface (or, in other words, there are no exotic ladder surfaces).
\end{abstract}

\maketitle


\section{Introduction}

A Riemann surface \( X \) is \( K \)-\emph{quasiconformally homogeneous}, or \( K \)-QCH, if given any two points \( x,y \in X \) there exists a \( K \)-quasiconformal homeomorphism \( f \co X \to X \) such that \( f(x) = y \). 
We say a Riemann surface is \emph{quasiconformally homogeneous}, or QCH, if it is \( K \)-QCH for some \( K \) (note: this definition diverges from the literature, where such a surface is usually referred to as \emph{uniformly} quasiconformally homogeneous). 
For a survey of the work on QCH surfaces see \cite{QCHSurvey}. 

For example, the Riemann sphere, the unit disk, and any Riemann surface whose universal cover is isomorphic to the complex plane are all 1-QCH, or \emph{conformally homogeneous}.
In fact, this is a complete characterization of conformally homogeneous Riemann surfaces, which leads us to the problem for which this paper is concerned:

\begin{quote}
\emph{Characterize all QCH Riemann surfaces.}
\end{quote}

Given the characterization of 1-QCH Riemann surfaces above, all the remaining cases to consider are hyperbolic Riemann surfaces (i.e.~Riemann surfaces whose universal cover is isomorphic to the unit disk).

The starting point for such a characterization comes from higher dimensions.
The notion of being \( K \)-QCH readily extends to the setting of hyperbolic manifolds of any dimension.
In dimension at least three, it was shown in \cite[Theorem 1.3]{BonfertQuasiconformal} that a hyperbolic manifold is QCH if and only if it is a (geometric) regular cover of a closed hyperbolic orbifold.
Naturally, such a result relies on rigidity phenomena in higher dimensions that do not occur in dimension two; in particular, as being QCH is invariant under quasiconformal deformations, it is not too difficult to find a hyperbolic QCH surface that does not regularly cover a closed hyperbolic orbifold (see \cite[Lemma 5.1]{BonfertQuasiconformal}).

This leads one to wonder---maybe naively---if every hyperbolic QCH surface is quasiconformally equivalent to a cover of a closed hyperbolic orbifold?
Interestingly, this is not the case: in \cite[Theorem 1.1]{BonfertExotic} the existence of \emph{quasiconformally exotic} QCH surfaces (i.e.~ QCH surfaces that are not quasiconformally equivalent to regular covers of closed orbifolds) is shown.
However, all the exotic QCH surfaces constructed in \cite{BonfertExotic} are homeomorphic; in particular, they are homeomorphic to the one-ended infinite-genus surface (affectionately referred to as the Loch Ness monster surface).

Our first theorem establishes that all QCH surfaces (and, in particular, exotic QCH surfaces) are topological regular covers of closed surfaces, or in other words,  there are no \emph{topologically exotic} QCH surfaces:

\begin{ThmA}
Every quasiconformally homogeneous Riemann surface topologically covers a closed surface.
\end{ThmA}

Note that every closed Riemann surface is QCH (see \cite[Proposition 2.4]{BonfertQuasiconformal} for a bound), so in the characterization of all QCH surfaces it is only left to consider non-compact surfaces. 
As a corollary to Theorem \ref{thm:regular}, we see that, up to homeomorphism, there are only a finite of number cases to consider.
In particular, combining Theorem \ref{thm:regular} with the classification of non-simply connected, infinite sheeted, regular covers of closed surfaces (Proposition \ref{prop:cover classification} below), we have:

\begin{cor}
\label{cor:top-exotic}
Up to homeomorphism, there are six non-compact  QCH Riemann surfaces, namely the plane, the annulus, the Cantor tree surface, the blooming Cantor tree surface, the Loch Ness monster surface, and the ladder surface\footnote{This nomenclature is explained in Proposition~\ref{prop:cover classification}}.
\end{cor}

As an immediate consequence of Corollary~\ref{cor:top-exotic}, we can strengthen a result of Kwakkel--Markovic \cite[Proposition 2.6]{KwakkelQuasiconformal}:

\begin{cor}
A Riemann surface of positive, finite genus is quasiconformally homogeneous  if and only if it is closed.
\end{cor}

Consider the non-hyperbolic cases in Corollary~\ref{cor:top-exotic}: we know that (1) every Riemann surface homeomorphic to the plane is QCH and (2) that a Riemann surface homeomorphic to the annulus is QCH if and only if its universal cover is isomorphic to \( \mathbb C \) (this follows from the discussion of 1-QCH surfaces, Theorem~\ref{thm:regular}, and the fact that the fundamental group of a closed hyperbolic Riemann surface does not have a cyclic normal subgroup---this also follows from \cite[Theorem 1.1]{BonfertQuasiconformal}).
This leaves only four topological cases to consider.

In this article, we give a characterization in one of the four cases: the ladder surface, that is, the two-ended infinite-genus surface with no planar ends.
In this case, our main theorem shows that there are no exotic QCH ladder surfaces, yielding a complete classification of QCH ladder surfaces:

\begin{ThmD}
A hyperbolic ladder surface is quasiconformally homogeneous if and only if it is quasiconformally equivalent to a regular cover of a closed hyperbolic surface.
\end{ThmD}

Given Theorem~\ref{thm:qch}, it is natural to ask if the distance (in the Teichm\"uller metric) of a \( K \)-QCH ladder surface from a regular cover can be explicitly bounded as a function of \( K \).
With this in mind, all of our proofs are written with the goal of providing explicit bounds in terms of \( K \) for all constants that appear; however, we are unable to do this in one location, namely the constant \( A \) appearing in Lemma~\ref{lem:area}.
It would be interesting to find such a bound.

The first step in the proof of Theorem \ref{thm:qch} is to choose a distance-minimizing geodesic (that is, a proper embedding of \( \mathbb R \) minimizing the distance between any two of its points); however, to do so, we need to know such a geodesic exists.
Our final theorem provides a sufficient topological condition for such a geodesic to exist; in addition, we show that there can be no topological condition that is both necessary and sufficient for a distance-minimizing geodesic to exist.

\begin{ThmE}
Every non-compact hyperbolic Riemann surface with at least two topological ends contains a distance-minimizing geodesic. 
Moreover, if an orientable, non-compact topological surface has a unique end, then it admits complete hyperbolic structures containing distance-minimizing geodesics as well as complete hyperbolic structures that do not admit such geodesics.
\end{ThmE}

Despite the narrative arc of the results above, in what follows, the proofs of the theorems will appear in reverse order. 

\subsection*{Acknowledgements}

The authors are grateful to Richard Canary and Hugo Parlier for helpful conversations. 

This project began several years ago during a visit of the second author to the first (before they were at the same institution) that was funded by the GEAR Network and so: the second author acknowledges support from U.S. National Science Foundation grants DMS 1107452, 1107263, 1107367 "RNMS: GEometric structures And Representation varieties" (the GEAR Network). 
During that time the second author was a postdoc at the University of Michigan and supported in part by NSF RTG grant 1045119.
The second author is currently supported in part by PSC-CUNY Award \#62571-00 50 and \#63524-00 51. The first author is supported  by  a grant from the Simons foundation
(359956, A.B.).


\section{Preliminaries}

Every Riemann surface is Hausdorff, orientable, and second countable; hence we will require these attributes of all topological surfaces in this note. 

\subsection{Hyperbolic geometry}
We mention some facts in hyperbolic geometry that will be used in the sequel. 
For more detailed information, see \cite{BuserBook,HubbardBook,Vuorinen,Bridson}. 

A homeomorphism \( f \co U \to V \) between domains \( U \) and \( V \) in \( \mathbb C \) is \( K \)-quasiconformal if 
\[
\frac1K \rm{Mod}(A) \leq \rm{Mod}(f(A)) \leq K \rm{Mod}(A)
\]
for any annulus \( A \) in \( U \), where \( \rm{Mod} (A) \) is the modulus of \( A \), that is, the unique real positive number \( M \) such that \( A \) is isomorphic to \( \{z \in \mathbb C : 1 < |z| < e^{2\pi M} \} \).
To extend to Riemann surfaces, we say a homeomorphism \( f \co X \to Y \) is \( K \)-quasiconformal if the restriction to any chart is \( K \)-quasiconformal. 

A Riemann surface \( X \) is \emph{hyperbolic} if its universal cover is isomorphic to the unit disk \( \mathbb D \).
We can then realize \( X \) as the quotient of \( \mathbb D \) by the action of a Fuchsian group \( \Gamma \).

If we equip \( \mathbb D \) with its unique Riemannian metric of constant curvature -1, then \( \Gamma \) acts on \( \mathbb D \) by isometries and this metric descends to a metric on \( X \), which will generally denote by \( \rho \). 
Given a closed geodesic \( \gamma \) in a hyperbolic Riemann surface \( X \), we let \( \ell_X(\gamma) \) denote its length in \( (X, \rho) \).

We first recall some basic geometric properties of quasiconformal maps.
Before doing so, we require some notation.
Given two compact subsets \( C_1 \) and \( C_2 \) in a metric space \( (M,d) \), let \( d(C_1,C_2) \) denote the distance between the two subsets, that is, 
\[
d(C_1,C_2) = \min\{ d(x,y) : x \in C_1, y \in C_2 \},
\]
and let \( H(C_1, C_2) \) denote the \emph{Hausdorff distance} between \( C_1 \) and \( C_2 \), that is,
\[
H(C_1, C_2) = \max\{ \sup_{x \in C_1} \inf_{y\in C_2} d(x,y), \sup_{y\in C_2} \inf_{x\in C_1} d(y,x) \}.
\]

Finally, given two metric spaces \( (M_1,d_1) \) and \( (M_2, d_2) \), a surjection \( f\co M \to N \) is an \( (A,B) \)-quasi-isometry if 
\[
\frac1A d_1(x,y) - B \leq d_2(f(x), f(y)) \leq A d_1(x,y) + B
\]
for all \( x,y \in M_1 \).

\begin{lem}
\label{lem:basics}
Let \( Z \) be a hyperbolic Riemann surface and let \( \gamma \) be a simple closed geodesic in \(  Z \).
If \( f \co Z \to Z \) is \( K \)-quasiconformal, then
\begin{enumerate}[(i)]
\item \( \frac{1}{K} \ell_Z (\gamma) \leq  \ell_Z (f(\gamma)) \leq K \ell_Z (\gamma) \),

\item \( f \) is a \( (K,K\log4) \)-quasi-isometry, and

\item there exists a constant \( R \) depending only on \( K \) and the length of \( \gamma \) so that \( H(f(\gamma), \delta) < R \), where \( \delta \) is the geodesic homotopic to \( f(\gamma) \). 
\end{enumerate}
\end{lem}

Throughout our arguments, we will require the use of the collar lemma:

\begin{thm}
Let \( X \) be a hyperbolic Riemann surface and let \( \eta\co \mathbb R \to \mathbb R \) be given by 
\[ \eta(\ell) = \arcsinh\left(\frac{1}{\sinh(\ell/2)}\right).\]

If \( \gamma_1 \) and \( \gamma_2 \) are disjoint simple closed geodesics of length \( \ell_1 \) and \( \ell_2 \), respectively, then the \( \eta(\ell_i) \)-neighborhood of \( \gamma_i \), that is, the set
\[
A_{\eta(\ell_i)}(\gamma_i) = \{ x \in X : \rho(x, \gamma) < \eta(\ell) \}
\]
is embedded in \( X \) and \( A_{\eta(\ell_1)}(\gamma_1) \cap A_{\eta(\ell_2)}(\gamma_2) = \emptyset \). 
\end{thm}

We end with a special property of compact hyperbolic surfaces with totally geodesic boundary---that is, a compact surface arising as the quotient  of a countable intersection of pairwise-disjoint closed half planes in \( \mathbb D \) by the action of a Fuchsian group. 
In a hyperbolic surface with totally geodesic boundary, an \emph{orthogeodesic} is geodesic arc whose end points meet the boundary of the surface orthogonally. 

A \emph{pants decomposition} of a topological surface is a collection of pairwise-disjoint simple closed curves, called the \emph{cuffs}, such that each complementary component of their union is homeomorphic to a thrice-punctured sphere.  
Every orientable topological surface with non-abelian fundamental group has a pants decomposition and, in a compact hyperbolic surface (possibly with boundary), there always exists a pants decomposition where the curves are of bounded length, with the bound depending only on the topology of the surface and the length of its boundary.

\begin{thm}[Bers pants decomposition theorem]
Given positive real numbers \( A \) and \( L \), there exists a positive real number \( B \)---depending only on \(  A \) and \( L \)---such that every compact hyperbolic surface with totally geodesic boundary whose boundary length is less than \( L \) and whose area is less than \( A \) admits a pants decomposition with cuff lengths bounded above by \( B \). 
\end{thm}

\subsection{Topological ends}

The notion of an end of a topological space was introduced by Freudenthal and, in essence, encodes the topologically distinct ``directions" of going to infinity in a non-compact space.  

More formally, for a non-compact second-countable surface \( S \), fix an exhaustion \( \{K_n\}_{n\in \mathbb N} \) of \( S \) by compact sets so that for each \( n \in \mathbb N \),  \( K_n \) lies in the interior of \( K_{n+1} \) and such that each component of the complement of \( K_n \) is unbounded.
We then define a \emph{ (topological) end}  of \( S \) to be a sequence \(  e = \{ U_n \}_{n\in\mathbb N} \), where \( U_n \) is a complementary component of \( K_n \) and \( U_n \supset U_{n+1} \).

The \emph{space of ends} of \( S \), denoted \( \mathcal E(S) \), is the set of ends of \( S \) equipped with the topology generated by sets of the form \( \hat U_n = \{ e \in \mathcal E(S) : U_n \in e\} \).
It is an exercise to check that, up to homeomorphism, the definition of \( \mathcal E(S) \) given does not depend on the choice of compact exhaustion.
We will say that an open subset \( V \) of \( S \) with compact boundary is a \emph{neighborhood} of an end \( e = \{U_n\}_{n\in\mathbb N} \) if there exists \( N \in \bn \) such that \( U_N \subset V \).  

We say an end \( e = \{U_n\}_{n\in\mathbb N} \) is \emph{planar} if, for some \( N \in \bn \), \( U_N \) is planar (i.e. homeomorphic to a subset of \( \br^2 \)).
We denote the set of non-planar ends by \( \mathcal E_{np}(S) \), which is a closed subset of \( \mathcal E(S) \).
Note that \( \mathcal E_{np}(S) \) is non-empty if and only if \( S \) has infinite genus.

\begin{thm}[Classification of surfaces (see \cite{RichardsClassification})]
Two orientable surfaces without boundary, \( S_1 \) and \( S_2 \), of the same (possibly infinite) genus are homeomorphic if and only if there is a homeomorphism \( \mathcal E(S_1) \to \mathcal E(S_2) \) sending \( \mathcal E_{np}(S_1) \) onto \( \mathcal E_{np}(S_2) \).
\end{thm}


\section{distance-minimizing geodesics and rays}

In a hyperbolic surface \( X \), a \emph{distance-minimizing geodesic} is a unit-speed geodesic curve \( \alpha \co  \br  \to X  \)  such that \( d_X(\gamma(a),\gamma(b)) = |b-a| \) for all \( a,b \in \mathbb R \).  
Recall that a map is \emph{proper} if the inverse image of a compact set is compact.

\begin{lem}
\label{lem:proper}
Every distance-minimizing geodesic in a hyperbolic Riemann surface is proper.
\end{lem}

\begin{proof}
Let \( X \) be a hyperbolic Riemann surface and let \( \alpha\co \mathbb R \to X \) be a continuous non-proper map.
Then, there exists a compact set  \( K \) such that \( \alpha^{-1}(K) \) is closed and not compact; in particular, \( \alpha^{-1}(K) \) is unbounded while \( K \) is bounded.
Hence, \( \alpha \) cannot be a distance-minimizing geodesic.  
\end{proof}

As an easy consequence of Lemma \ref{lem:proper}, no compact hyperbolic Riemann surface can have a distance-minimizing geodesic.
In Theorem \ref{thm:arcs}, we give a topologically sufficient condition for distance-minimizing geodesics to exist in a non-compact hyperbolic Riemann surface; however, in addition, we see that there cannot be a necessary topological condition for the existence of such a geodesic.

\begin{thm}
\label{thm:arcs}
Every non-compact hyperbolic Riemann surface with at least two topological ends contains a distance-minimizing geodesic. 
Moreover, if an orientable, non-compact non-planar topological surface has a unique end, then it admits complete hyperbolic structures containing distance-minimizing geodesics as well as complete hyperbolic structures that do not admit such geodesics.
\end{thm}

We split the proof into three lemmas covering the separate cases.
Let us first consider the multi-ended case.

\begin{lem}
\label{lem:two}
Let \( X \) be a hyperbolic Riemann surface.
If \( X \) has at least two topological ends, then \( X \) contains a distance-minimizing geodesic with distinct ends. 
\end{lem}

\begin{proof}
Let \( e_1 \) and \( e_2 \) be distinct topological ends of \( X \) and let \( \eta \) be a separating, simple, closed geodesic separating \( e_1 \) and \( e_2 \).
Label the  two components of \( X \ssm \eta \) by \( U_1 \) and \( U _2 \) so that \( U_i \) is a neighborhood of \( e_i \).
For \( i\in \{1,2\} \), choose a sequence \( \{x_n^i\}_{n\in\bn} \) in \( U_i \) such that \( \lim x_n^i = e_i \). 
Let \( \gamma_n\co I_n \to X \) be the minimal-length unit-speed geodesic curve between \( x_n^1 \) and \( x_n^2 \). 
Observe that \( \gamma_n \) intersects \( \eta \) (exactly once) for each \( n \in \bn \); let \( t_n \in I_n \) such that \( \gamma_n(t_n) \in \eta \).
By the compactness of the lift of \( \eta \) to the unit tangent bundle of \( X \), the sequence \( \{\gamma_n'(t_n)\}_{n\in\bn} \) accumulates; let \( v \) be an accumulation point of the sequence 
and let \( \al \) be the corresponding unit-speed geodesic in \( X \).
Note that the endpoints of \( \al \) are \( e_1 \) and \( e_2 \) and hence they are distinct.

We claim \( \al \) is a distance-minimizing geodesic: assume not and let \( w \) and \( z \) be points on \( \alpha \) such that there exists a distance-minimizing path \( \delta \) of length strictly less than that of the segment of \( \al \) connecting \( w \) and \(z \).
Let \( \beta \) denote the segment of \( \al \) connecting \( w \) and \( z \) and let \( \Delta = \ell(\beta) - \ell(\delta) \). 
Choose a positive real number \( \ep \) so that \( \ep < \Delta \) and such that the \( \frac \ep2 \)-neighborhood \( Q \) of \( \beta \)  (that is, \( Q = \{ x \in X : \rho(x,\delta) < \frac\ep2\} \)) is isometric to the \( \frac\ep2 \)-neighborhood of a geodesic segment in \( \mathbb{H} \) of length \( \ell(\beta) \).
Note that \( Q \) has two geodesic sides, one containing \( w \) and the other containing \( z \).

Now there exists \( N \in \bn \) such that \( \gamma_N \cap Q \) is connected with endpoints \( w_N \) and \( z_N \) on the same sides of \( Q \) as \( w \) and \( z \), respectively. 
If \( \delta_w \) and \( \delta_z \) are the shortest curves in \( Q \) connecting \( w_N \) to \( w \) and \( z_N \) to \( z \), respectively, then, as \( \delta_w \cup \delta \cup \delta_z \) is a path connecting \( w_N \) and \( z_N \), it follows that
\[
\ell(\delta_w) + \ell(\delta) + \ell(\delta_z) > \ell(\gamma_N \cap Q) > \ell(\beta),
\]
where the first inequality follows from the fact that \( \gamma_N \) is distance minimizing and the second follows from \( \beta \) being the orthogonal connecting the geodesic sides of \( Q \).
But, at the same time, we have 
\[
\ell(\delta_w) + \ell(\delta) + \ell(\delta_z) < \ell(\delta) + \ep < \ell(\delta) + \Delta = \ell(\beta).
\]
However, this is a contradiction as both inequalities cannot hold.
\end{proof}

We now move to the one-ended case.
For Lemma \ref{lem:one-no} and Lemma \ref{lem:one-yes} below, we remind the reader that, up to homeomorphism, there is a unique one-ended, orientable surface whose end is non-planar, namely, the Loch Ness monster surface.

\begin{lem}
\label{lem:one-no}
If \( S \) is a non-planar one-ended, orientable surface, then there exists a hyperbolic Riemann surface \( X \) homeomorphic to \( S \) that does not contain a distance-minimizing geodesic. 
\end{lem}

\begin{proof}
There are two cases: either the end of \( S \) is planar or not.
Since \( S \) has positive genus, if the end of \( S \) is planar, we can choose a hyperbolic Riemann surface \( X \) homeomorphic to \( S \) in which the end of \( S \) corresponds to a cusp on \( X \).
Let \( \al \co \mathbb R \to X \) be a continuous function.
If \( \al \) fails to be proper, then it is not distance minimizing by Lemma \ref{lem:proper}, so we may assume that \( \al \) is proper.
In this case, the two unbounded components of the intersection of  \( \alpha \) with a cusp neighborhood become arbitrarily close; hence, \( \alpha \) cannot be a distance-minimizing geodesic.

Now suppose that the end of \( S \) is non-planar.
Let \( c_1 \) be any separating, simple closed curve in \( S \).
We inductively build a sequence of disjoint, separating, simple closed curves \( \{c_n\}_{n\in\bn} \) by requiring that \( c_{n+1} \) separates \( c_n \) from the end of \( S \).
Let \( X \) be a hyperbolic Riemann surface such that there exists \( L > 0 \) so that the length of the geodesic representative \( \gamma_n \) of \( c_n \) has length less than \( L \).  
Now let \( \al\co \mathbb R \to X \) be a geodesic in \( X \); as before, we may assume that \( \al \) is proper.
By the length restriction on the \( \gamma_n \), there exists some \( N \in \mathbb N \) such that the intersection of \( \alpha \) with the bounded component of \( X \ssm \gamma_N \) has length greater than \( L \); hence, \( \alpha \) cannot be a distance-minimizing geodesic. 
\end{proof}

\begin{lem}
\label{lem:one-yes}
If \( S \) is a borderless one-ended, orientable surface, then there exists a hyperbolic Riemann surface \( X \) homeomorphic to \( S \) containing a distance-minimizing geodesic.
\end{lem}

\begin{proof}
Again we split into two cases: first, suppose that the end of \( S \) is planar. 
In this case, let \( X \) be a hyperbolic Riemann surface homeomorphic to \( S \) such that the end of \( S \) corresponds to a funnel on \( X \).
All funnels have distance-minimizing geodesics.

\begin{figure}[t]
\centering
\includegraphics[scale=0.9]{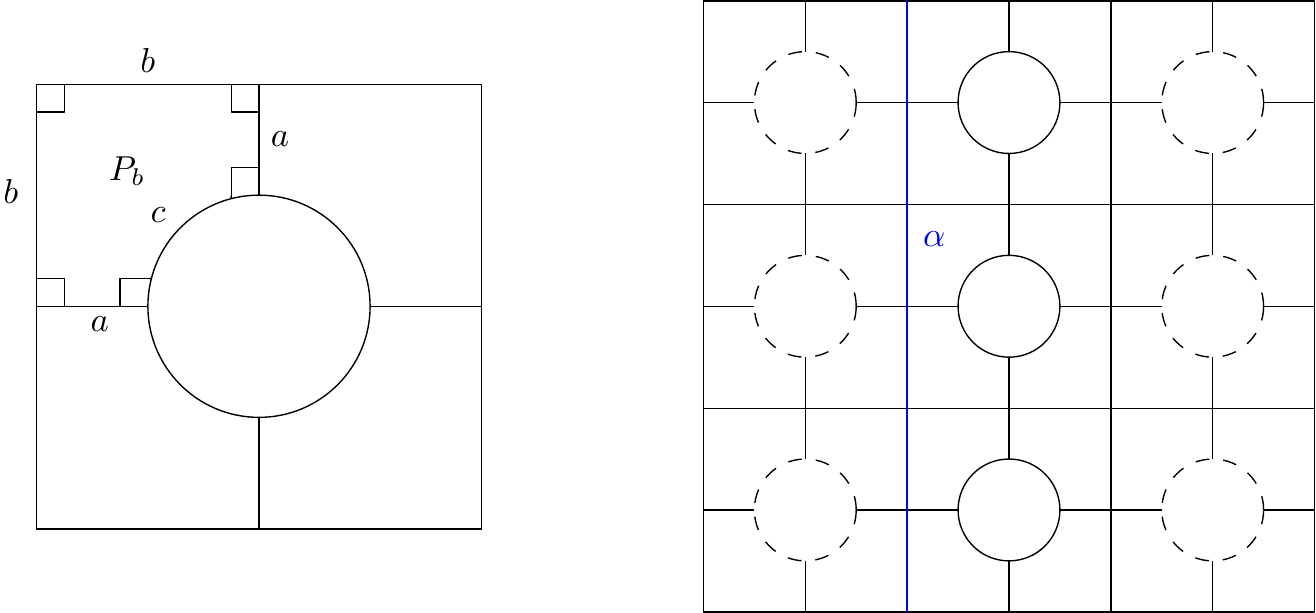}
\caption{On the left, \( P \) is a hyperbolic right-angled pentagon and four copies of \( P \) are glued to form a square with a disk removed.
On the right, these 1-holed squares are glued in a tiling extending in all directions; the extension of the vertical geodesic \( \al \) forms a proper geodesic arc.}
\label{fig:hairytorus}
\end{figure}

Let us continue to the case where \( S \) has a non-planar end.  
Given a positive number \( b \), there exists a unique right-angled hyperbolic pentagon \( P_b \) having two consecutive sides of length \( b \); let the consecutive sides of \( P_b \) have lengths \( b,b,a,c,a \) as in Figure~\ref{fig:hairytorus}.
We can glue four copies of \( P_b \) to form a 1-holed squared with outer boundary having length \( 8b \) and inner boundary \( 4c \). 

We now build a bordered hyperbolic surface inductively: let \( T^1_b \) be a copy of the 1-holed square above.
For \( n \in \bn \), we construct \( T^{n+1}_b \) by pasting eight copies of \( T^n_b \) to form a rectangle with \( 9^{n-1} \) holes.
We identify \( T^n_b \) with the middle copy of \( T^n_b \) in \( T^{n+1}_b \).
We then let \( T_b \) be the direct limit of the \( T^n_b \).
(Less formally, we obtain \( T_b \) by tiling the plane with copies of the 1-holed square, see Figure \ref{fig:hairytorus}.)

Let \( R_b \) be the hyperbolic Riemann surface obtained by identifying the boundary components of \( T_b \) horizontally via an orientation-reversing isometry, so that a (dashed) circle is identified with the (solid) circle to its right in Figure \ref{fig:hairytorus}.
Note that \( R_b \) is infinite genus and one-ended and hence homeomorphic to \( S \).

Now let \( \al \) be a vertical geodesic as in Figure \ref{fig:hairytorus}. 
We claim that \( \al \) is a distance-minimizing geodesic in \( R_b \).
It is enough to prove that \( \al \) minimizes distances between the corners of 1-holed squares for which it passes.
Let \( x \) and \( y \) be two such corners and let \( \gamma \) a distance-minimizing path between them.
By construction, the shortest path from one side of a 1-holed square to any other is at least \( 2b \). 
It follows that distance between two infinite horizontal geodesics in \( T_b \) is exactly \( 2b \). 
The same is true in \( R_b \) as the gluing of boundary components does not change height.
Now if \( \al \) crosses through \( n-1 \) horizontal geodesics from \( x \) to \( y \), then \( \gamma \) must do the same and in particular the length of \( \gamma \) is at least \( 2nb \), which of course is the length of the segment of \( \al \) connecting \( x \) and \( y \).
\end{proof}

Of course the difficulty in the one-ended case is that a proper arc needs to approach the unique end of the surface in both the forwards and backwards directions. 
To capture this, we prove the existence of a distance-minimizing ray.
Here, a \emph{ray} is the image of a continuous injective map of the half line \( [0, \infty) \subset \br \). 

\begin{prop}
Every point on a non-compact hyperbolic Riemann surface is the base point of some distance-minimizing ray.
\end{prop}

\begin{proof}
We provide the sketch of the proof as the details are nearly identical to those in the proof of Lemma \ref{lem:two}.
Let \( X \) be a non-compact hyperbolic Riemann surface and let \( e \) be a topological end of \( X \).
Fix a sequence \( \{x_n\}_{n\in\bn} \) that limits to \( e \).
Now let \( x \) be a point in \( X \) and, for \( n \in \bn \), let \( \gamma_n \) denote a unit-speed, minimal-length, geodesic curve starting at \( x \) and ending at \( x_n \).
Let \( v_n = \gamma_n'(0) \), then we may choose a unit vector \( v \) in the accumulation set of the sequence \( \{v_n\}_{n\in\bn} \).
Arguing as in Lemma \ref{lem:two}, the geodesic ray based at \( x \) determined by \( v \) is distance minimizing. 
\end{proof}


\section{QCH ladder surfaces are regular covers}

In this section, we prove our main theorem:

\begin{thm}
\label{thm:qch}
A hyperbolic ladder surface is quasiconformally homogeneous if and only if it is quasiconformally equivalent to a regular cover of a closed hyperbolic surface.
\end{thm}

It is not difficult to see that being QCH is a quasi-conformal invariant and that every regular cover of a closed hyperbolic surface is QCH (see \cite[Proposition 2.7]{BonfertQuasiconformal}); hence, to prove Theorem~\ref{thm:qch}, we only need to focus on the forwards direction.

The proof will be split into the lemmas in the subsections below.
Throughout the subsections below $X$ denotes a \( K \)-QCH ladder surface. 
Let  $\mathcal{F}_K$  be the set of $K$-quasiconformal homeomorphisms \( X \to X \).
We say that a simple closed curve in \( X \) \emph{separates the ends} of \( X \) if its complement consists of two unbounded components. 

\subsection{Shiga pants decomposition}

A \emph{Shiga pants decomposition} of a hyperbolic Riemann surface is a pants decompositions whose cuff lengths are uniformly bounded from above. 
The goal of this subsection is to show that every QCH ladder surface has a Shiga pants decomposition.
It seems natural to expect a QCH surface to have such a pants decomposition, but, in fact, this is not always the case.
For example, the surface \( R_b \) constructed in the proof of Lemma \ref{lem:one-yes} is QCH (it is a regular cover of a closed hyperbolic surface), but does not have a Shiga pants decomposition \cite{ParlierPrivate}.

The first step in the proof is to find a sequence of pairwise-disjoint simple closed geodesics that separate the ends of \( X \), that are of uniformly bounded length, and that are ``evenly" spaced throughout the surface (Lemma \ref{lem:special curves}).
We then show that the subsurfaces in the complement of these curves have bounded topology (this will follow from Lemma~\ref{lem:area}); the existence of a Shiga pants decomposition for \( X \) will follow by taking a Bers pants decomposition for each of these complementary subsurfaces.

Two real-valued functions $f(x)$ and $g(x)$ are said to be {\it comparable}, denoted $f \asymp g$, if there exists positive constants $A$ and $B$ so that
$A \leq \frac{f(x)}{g(x)} \leq B$,  for all $x$.

\begin{lem} 
\label{lem:special curves}
There exists a sequence of pairwise-disjoint simple separating geodesics \( \{\gamma_n\}_{n\in\mathbb Z} \) such that each \( \gamma_n \) separates the ends of \( X \) and so that
\begin{equation} \label{ }
\rho (\gamma_n, \gamma_{m}) \asymp H(\gamma_n, \gamma_{m})
\asymp |m-n|
\end{equation}
for all $n,m  \in \mathbb{Z}$. 
Moreover, the constants in the comparisons depend only on  $K$  and $L = \ell_X (\gamma_0)$.
\end{lem}

\begin{proof}
Choose any simple closed geodesic separating the ends of \( X \) and label it \( \gamma_0 \).
Set $L = \ell_X(\gamma_0)$. 
As \( X \) is two ended, by Lemma \ref{lem:two}, we may choose a distance-minimizing geodesic \( \beta \) on \( X \) with distinct ends. 
Identify \( \beta \) with a unit-speed parameterization \( \beta \co \mathbb R \to X \) such that \( \beta(0) \in \gamma_0 \); set \( x_0 = \beta(0) \).

Let \( R = R(K) \) be as in Lemma \ref{lem:basics} and, for \( n \in \mathbb Z \), let \( x_n = \beta( 3n(R+KL) ) \). 
As \( X \) is \( K \)-QCH, we may choose $f_n \in \mathcal{F}_K$ such that $f_n(x_0)=x_n$.
Finally, let   $\gamma_n$ be the geodesic  in the homotopy class of $f_n(\gamma_0)$.
(Recall \( H(\gamma_n, f_n(\gamma_0)) \leq R \).)

We claim that the sequence $\{\gamma_n\}_{n\in\mathbb Z}$  has the desired properties. 
To see this, first observe, for every \( n \in \mathbb Z \), that \( \ell_X(\gamma_n) \leq KL \) and \( \gamma_n \) separates the ends of \( X \).
Next, we compute the distance between $\gamma_n$ and $\gamma_{n+1}$.  
Observe that we can construct a path from \( x_n \) to \( x_{n+1} \) of length less than \( \frac{KL}2 + R + \rho(\gamma_n, \gamma_{n+1}) + R + \frac{KL}2 \), which must have length at least \( 3(R+KL) \) as \( \beta \) is distance minimizing; hence,
\[
\rho(\gamma_n, \gamma_{n+1}) \geq 3(R+KL)-2R-KL=R+2KL.
\]
Regarding an upper bound, we have
\[
\rho(\gamma_n, \gamma_{n+1}) \leq \rho(x_n,x_{n+1}) + 2 R = 3(R+KL) + 2R = 5R + 3KL,
\]
where the first inequality uses \( H( f(\gamma_n), \gamma_n ) < R\).

Assume $m>n$ and recall that any path from \( \gamma_n \) to \( \gamma_m \) must pass through \( \gamma_k \) for all \( n < k < m \) and hence

\begin{equation}
\label{eq:rho_lower}
\rho(\gamma_n, \gamma_{m}) 
\geq  \sum_{k=n}^{m-1} \rho(\gamma_k, \gamma_{k+1})
\geq  (m-n)(R+2KL).
\end{equation}

It follows that \( \rho(\gamma_n,\gamma_m) > 0 \); in particular, \( \gamma_n \cap \gamma_m = \emptyset \) for all distinct \( n,m \in \mathbb Z \). 
Now, as \( \gamma_n \) and \( \gamma_m \) are disjoint,  \( \rho(\gamma_n, \gamma_m) \) is realized by an orthogeodesic between them.
For the upper bound, using the fact that the orthogeodesic from 
$\gamma_n$ to $\gamma_m$ is shorter than the piecewise-continuous curve made up of orthogeodesics between successive $\gamma_k$ and arcs  along the $\gamma_k$ we have

\begin{equation}
\label{eq:rho_upper}
\rho(\gamma_n, \gamma_{m}) 
\leq  \sum_{k=n}^{m-1} \left[ \rho(\gamma_k, \gamma_{k+1})
+\frac{\ell_X (\gamma_k)}{2}  \right]
\leq  (m-n)\left[ (5R+3KL) +  \frac{KL}{2}  \right]
\end{equation}
where the last inequality uses the fact that  
$\ell_X (\gamma_k) \leq KL$.
Combining \eqref{eq:rho_lower} and \eqref{eq:rho_upper}, we have shown

\begin{equation} 
\label{eq: hyp-combinatorial comparison}
(R+2KL)\leq \frac{\rho(\gamma_n, \gamma_{m})}{|m-n|}
\leq (5R+\frac{7KL}{2})
\end{equation}
implying \( \rho(\gamma_n, \gamma_m) \asymp |n-m| \).

 To show that  $\rho(\gamma_n, \gamma_{m})$ is comparable to 
$H(\gamma_n, \gamma_{m})$, we first consider the following inequality:

\begin{equation}
\rho(\gamma_n, \gamma_{m}) \leq H(\gamma_n, \gamma_{m})
\leq   \ell_X (\gamma_n)/2 +  \rho(\gamma_n, \gamma_{m}) + \ell_X (\gamma_m)/2 \leq KL+  \rho(\gamma_n, \gamma_{m})
\end{equation}

Dividing the above  inequality  by  
$\rho(\gamma_n, \gamma_{m})$ we obtain

\begin{equation}
\label{eq: inequality}
1 \leq \frac{H(\gamma_n, \gamma_{m})}{\rho(\gamma_n, \gamma_{m})} \leq \frac{KL}{\rho(\gamma_n, \gamma_{m})}+1
\end{equation}

For every integer \( n \), using that \( \gamma_n \) is the geodesic homotopic to \( f_n(\gamma_0) \), we have that \( \ell_X(\gamma_n) \geq \frac LK \).
Therefore, since \( \gamma_n \) and \( \gamma_m \) are disjoint for distinct integers \( n \) and \( m \),  the collar lemma implies that the \( \eta\left(\frac LK \right) \)-neighborhoods of \( \gamma_n \) and \( \gamma_m \) are embedded and disjoint.
In particular, we have \( \rho(\gamma_n, \gamma_m) \geq 2 \eta\left(\frac LK\right) \).
Therefore,  \eqref{eq: inequality} becomes

\begin{equation} \label{eq: haus-hyp comparison}
1 \leq \frac{H(\gamma_n, \gamma_{m})}{\rho(\gamma_n, \gamma_{m})} \leq \frac{KL}{2 \eta(\frac{L}{K})}+1.
\end{equation}

This finishes the proof of the lemma. 
\end{proof}

We remark that putting together lines \eqref{eq: hyp-combinatorial comparison}
and \eqref{eq: haus-hyp comparison} yields the concrete comparison:

\begin{equation} \label{eq: haus-combinatorial comparison}
(R+2KL) \leq \frac{H(\gamma_n, \gamma_{m})}{|m-n|} \leq 
\left(\frac{KL}{2 \eta(\frac{L}{K})}+1\right)\left(5R+\frac{7KL}{2}\right)
\end{equation}

\begin{lem}
\label{lem:area}
Let \( \{\gamma_n\}_{n\in\mathbb N} \) be the sequence of geodesics constructed in Lemma \ref{lem:special curves} and let \( Y_n \) be the compact subsurface co-bounded by \( \gamma_n \) and \( \gamma_{n+1} \).
There exists a positive real number \( A \)---depending on \( X \)---such that the area of \( Y_n \) is at most \( A \). 
\end{lem}

\begin{proof}
Denote  the lower bound of \eqref{eq: hyp-combinatorial comparison}  by $a$ and the upper bound of (\ref{eq: haus-combinatorial comparison})  by \( b \).  
Let \( C = K\log4 \), then, as $f \in \mathcal{F}_K$ is a \( (K,C) \)-quasi-isometry,
\begin{equation}\label{eq: dist between geodesics}
\rho(f (\gamma_n), f (\gamma_m)) \geq \frac{1}{K} \rho( \gamma_n, \gamma_m)-C
\end{equation}
for all \( n, m \in \mathbb Z \).

Choose \( m \in \mathbb{N} \) satisfying \( m > \frac Ka(b+C+R) \), where \( R \) is as in Lemma \ref{lem:basics}, and consider the geodesic subsurface $Z$ bounded by the 
geodesics $\gamma_{-m}$ and $\gamma_{m}$.
We set \( A \) to be the area of \( Z \).

Let \( f_n \in \mathcal F_k \) be as defined as in the proof of Lemma \ref{lem:special curves}.
We claim that $Y_n \subset f_n(Z)$ for all  $n \in \mathbb{Z}$.
Before proving our claim, note that the area of \( f_n(Z) \) is bounded above independent of \( n \in \mathbb{Z} \)---the bound only depends on \( K \) and the area of \( Z \).
To see this, let \( Z_n \) denote the geodesic straightening of \( f_n(Z) \), so that the area of \( Z_n \) agrees with that of \( Z \).
It must be that \( f_n(Z) \) is in the \( R \)-neighborhood of \( Z_n \).
The area of the \( R \)-neighborhood of \( Z_n \) is bounded above by the area of \( Z_n \) and \( R \).
As \( R \) only depends on \( K \), we see that the area of \( f_n(Z) \) is bounded as a function of \( K \) and the area of \( Z \). 

Now to prove the claim, first note that using (\ref{eq: dist between geodesics})  we have:

\begin{equation}
\label{eq:Z}
\rho(\gamma_n, f_n(\gamma_{m})) \geq \rho(f_n(\gamma_0),f_n(\gamma_m)) - R
\geq \frac{1}{K}\rho(\gamma_0,\gamma_m)-C-R
\geq  \frac{ma}{K}-C-R > b,
\end{equation}
where the last inequality follows from replacing \( m \) with the assumed lower bound in the choice of \( m \).

On the other hand, $H(\gamma_n, \gamma_{n+1}) < b$ and thus,
$\rho(\gamma_n, f_n(\gamma_m)) >  H(\gamma_n, \gamma_{n+1})$; in particular, \( f_n(\gamma_m ) \) must be disjoint from \( Y_n \). 
Observe that \eqref{eq:Z} holds with \( m \) replaced by \( -m \); hence, \( Y_n \) and \( f_n(\gamma_{-m}) \) are disjoint.
Thus  $Y_n \subset f_n (Z)$ and  hence the area of \( Y_n \) is less than \( A \).
\end{proof}

\begin{rem} Note that  \( m \) can be explicitly chosen to be a function of \( K \) and \( L \).  
\end{rem}

A Bers pants decomposition of each \( Y_n \) together with 
$\{\gamma_n\}_{n\in\mathbb Z}$ yields a pants decomposition for \( X \) with bounded cuff lengths, establishing: 

\begin{prop}
\label{prop:Shiga}
Every QCH ladder surface admits a Shiga pants decomposition.
\qed
\end{prop}

\subsection*{Aside: coarse geometry}
We take a short tangent from the proof of Theorem \ref{thm:qch} to discuss the coarse geometry of QCH ladder surfaces.

Note that as \( X \) is \( K \)-QCH, there is a lower bound on the injectivity radius of \( X \) depending only on \( K \) \cite[Theorem 1.1]{BonfertQuasiconformal}.
In particular, the diameter of \( Y_n \) is bounded above independent of \( n \).
Now let \( \beta \) be the distance-minimizing geodesic from Lemma \ref{lem:special curves}.
Define \( r \co X \to \beta \) by sending a point \( x \) of \( X \) to any point \( y \in \beta \) satisfying \( \rho(x,y) = \rho(x,\beta) \). 
It then follows that \( r \) is a quasi-isometry and hence \( \beta \) is a \emph{quasi-retract} of \( X \).
Note that in the proof of Lemma \ref{lem:special curves}, we could have chosen any distance-minimizing curve, establishing:

\begin{prop}
\label{prop:retract}
If \( X \) is a QCH ladder surface, then every distance-minimizing geodesic in \( X \) is a quasi-retract of \( X \).
\qed
\end{prop}

\begin{cor}
Every QCH ladder surface is quasi-isometric to \( \mathbb \br \) equipped with the standard Euclidean metric. 
\qed
\end{cor}

Let us show that Proposition \ref{prop:Shiga} and Proposition \ref{prop:retract} do not characterize the property of being QCH amongst hyperbolic ladder surfaces.

First, we note again that there exist QCH surfaces that do not admit a Shiga pants decomposition.

\begin{prop}
There exists a hyperbolic ladder surface \( R \) and distance-minimizing geodesic \( \beta \) in \( R \) such that \( \beta \) is a quasi-retract of \( R \), \( R \) admits a Shiga pants decomposition, and \( R \) is not QCH.
\end{prop}

\begin{proof}
Let \( S \) be a topological ladder surface and let \( \{ a_n, b_n, c_n \}_{n\in\mathbb Z} \) be a pants decomposition for \( S \) as in Figure~\ref{fig:pants-decomp2}.
Let \( R \) a hyperbolic surface such that \( \ell(a_n)=\ell(b_n) = 1 \) and \( \ell(c_n) = \left|\frac1n\right| \)  for all \( n \in \mathbb Z \), and such that the geodesic arc in the homotopy class of \( \beta \) is distance minimizing.

\begin{figure}[t]
\centering
\includegraphics{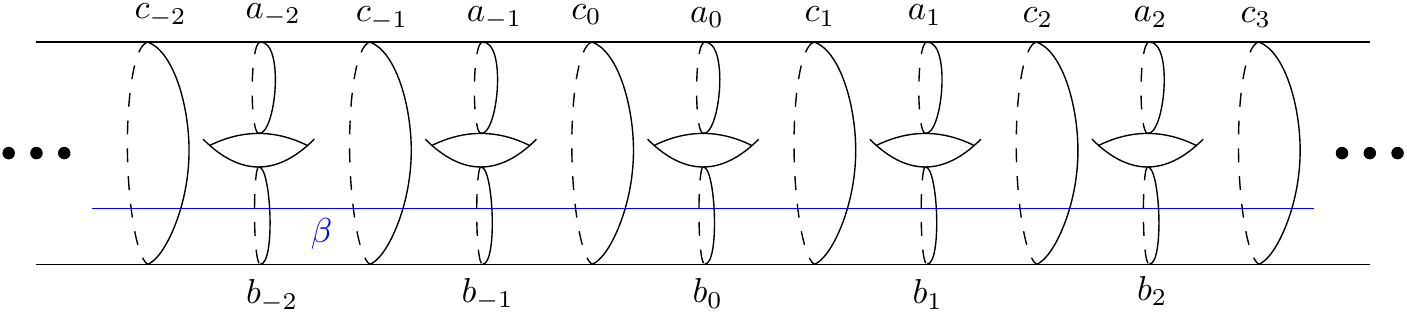}
\caption{A pants decomposition for a (topological) ladder surface.}
\label{fig:pants-decomp2}
\end{figure}

By construction, \( R \) has a Shiga pants decomposition and moreover the nearest point projection \( r \co R \to \beta \) is a quasi-isometry.
However, the injectivity radius of \( R \) goes to 0 and hence \( R \) cannot be QCH.
\end{proof}

\begin{prop}
There exists a hyperbolic ladder surface with a Shiga pants decomposition that is not quasi-isometric to \( \mathbb Z \) (and hence not QCH).
\end{prop}

\begin{proof}
Let \( \Gamma \) be the two-ended graph shown here:

\begin{center}
\includegraphics[scale=0.9]{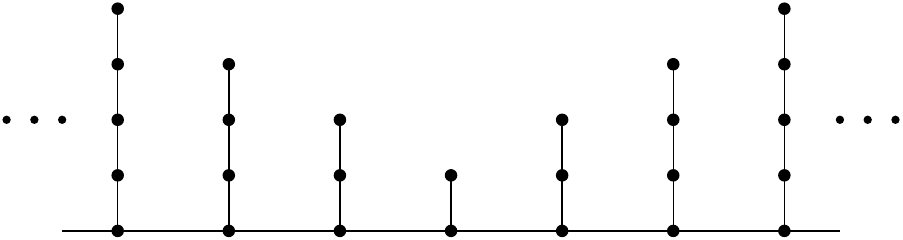}
\end{center}

For \( i \in \{1,2,3\} \), let \( V_i \) be a hyperbolic \( i \)-holed torus with each boundary component of length 1.
Let \( R \) be a hyperbolic surface obtained from \( \Gamma \) by taking a copy of \( V_i \) for each valence \( i \) vertex  and identifying boundary components according to the edge relations in \( \Gamma \).
The resulting surface \( R \) is quasi-isometric to \( \Gamma \); in particular, it is a ladder surface with a Shiga pants decomposition.
However, \( \Gamma \) and hence \( R \) is not quasi-isometric to \( \mathbb Z \).
\end{proof}

These propositions lead us the follow question, which we end the aside with:

\begin{question}
If a hyperbolic ladder surface has  positive injectivity radius and is quasi-isometric to \( \mathbb R \), then is it QCH?
\end{question}

\subsection{Preferred Shiga pants decomposition}
In the previous section, we showed that \( X \) admits a Shiga pants decomposition; however, this is just an existence statement and does not give us enough information to directly construct the desired covering map.
The goal of this subsection is to modify the Shiga pants decomposition from Proposition \ref{prop:Shiga} into a (topological) form we can use to build a deck transformation (the desired form is shown in Figure~\ref{fig:pants-decomp2}).

It is not difficult to show the existence of the desired pants decomposition using a continuity and compactness argument in moduli space; however, it is not possible to extract explicit length bounds from such an argument.
With a little extra effort, we proceed in a fashion allowing for effective constants---this is the content of Lemma~\ref{lem:short-pants}.

\begin{figure}[t]
\includegraphics[scale=0.9]{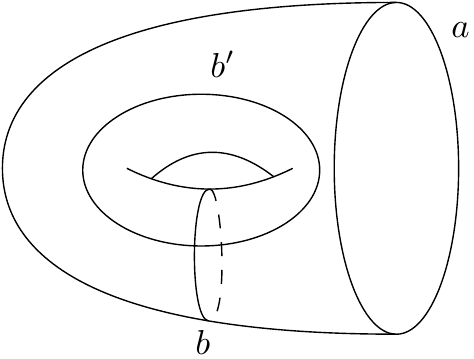}
\caption{The elementary move up to homeomorphism in a once-punctured torus switching \( b \) and \( b' \)}
\label{fig:move1}
\bigskip
\includegraphics[scale=0.9]{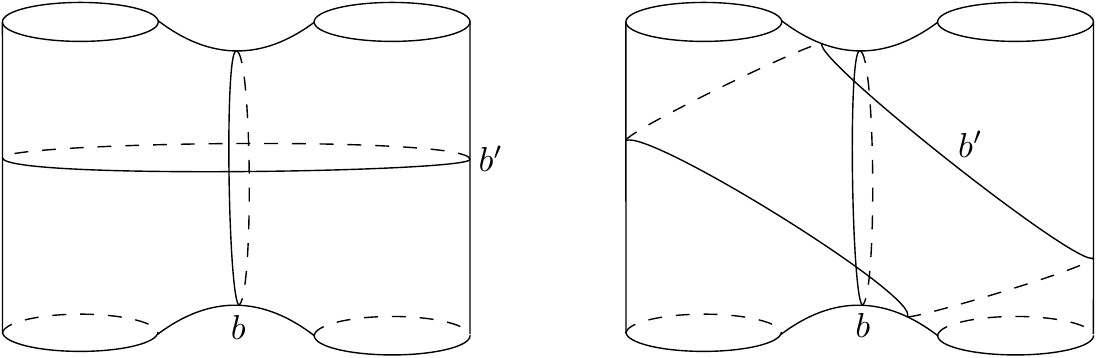}
\caption{The two possible types of elementary moves up to homeomorphism in a 4-holed sphere.}
\label{fig:move23}
\end{figure}

Let \( \Sigma \) be a compact surface with non-abelian fundamental group.
The \emph{pants graph} associated to \( \Sigma \), written \( \mathcal P(\Sigma) \), is the graph whose vertices correspond to pants decompositions of \( \Sigma \) (up to isotopy) and where two vertices are adjacent if they differ by an elementary move.
An elementary move corresponds to removing a single curve \( \alpha \) from the pants decomposition and replacing \( \alpha \) with a curve that is disjoint from all remaining curves of the pants decomposition and intersecting \( \alpha \) minimally (see Figures \ref{fig:move1} and \ref{fig:move23} ). 

Define an equivalence relation \( \sim \) on the vertices of \( \mathcal P(\Sigma) \) by setting two pants decompositions to be equivalent if they differ by a homeomorphism of \( \Sigma \).
The \emph{modular pants graph}, written \( \mathcal{MP}(\Sigma) \), is the graph whose vertices correspond to equivalence classes of pants decompositions of \( \Sigma \); two vertices are connected by an edge if they have representatives in \( \mathcal P (\Sigma) \) that are adjacent.

As \( \mathcal P(\Sigma) \) is connected \cite{HatcherPresentation}, we have that \( \mathcal{MP}(\Sigma) \) is connected; moreover, \( \mathcal{MP}(\Sigma) \) has finitely many vertices and hence finite diameter (in the graph metric). 
Observe that, up to homeomorphism, there are at most two ways to replace a single curve in a given pants decomposition; in particular, none of the edges in the modular pants graph correspond to the elementary move shown in Figure \ref{fig:move1}.

Let \( \sigma \) be a hyperbolic metric on \( \Sigma \).
Given a vertex \( v \in \mathcal {MP}(\Sigma) \), define
\begin{align*}
M_\sigma(v) = \min\{ M : & \text{ there exists } \{c_1, \ldots, c_\xi\} \in \mathcal P(\Sigma)  \text{ such that } \\ & [ \{c_1, \ldots, c_\xi\}] = v \text{ and } \ell_\sigma(c_i) \leq M \text{ for all } i \in \{1, \ldots, \xi\} \},
\end{align*}
where \( \xi = \xi(\Sigma) = 3g-3+b \) is the \emph{topological complexity} of \( \Sigma \) (\( g \) is the genus of \( \Sigma \) and \( b \) the number of boundary components of \( \Sigma \)). 

\begin{lem}
\label{lem:short-pants}
Let \( \Sigma \) be a compact surface possibly with boundary and with \( \xi(\Sigma) > 0 \).  
Let \( \sigma \) be a hyperbolic metric on \( \Sigma \) with injectivity radius \( m \), let \( v \in \mathcal{MP}(\Sigma) \), and let \( M \in \mathbb R \) such that \( M_\sigma(v) \leq M \).
Then, for all  \( w \in \mathcal{MP}(\Sigma) \), \( M_\sigma(w) \) is bounded above by a function of \( \xi, m, \) and \( M \).
\end{lem}

\begin{proof}
Let \( P \) be a pair of pants representing \( v \) with cuff lengths all bounded above by \( M \).
If \( \ell \) is the length of any orthogeodesic connecting a boundary component to itself, then 
\[
\sinh(\ell/2) \leq \frac{\cosh(M/2)}{\sinh(m/2)}
\] 
(this follows from standard formulas for right-angled hyperbolic pentagons, see \cite[Formula 2.3.4(1)]{BuserBook}).

Now choose a representative pants decomposition \( \tilde v = \{c_1, \ldots , c_\xi\} \) for \( v \) such that \( \ell_\sigma(c_i) \leq M \) for all \( i \in \{c_1, \ldots, c_\xi\} \).
Let \( w \in \mathcal{MP}(\Sigma) \) be adjacent to \( v \).
Up to relabelling, we can assume that there exists a representative \( \tilde w = \{c_1', c_2, \ldots, c_\xi\} \) of \( w \) adjacent to \( \tilde v \).

Let \( R \) denote the 4-holed sphere component of \( \Sigma \smallsetminus \bigcup_{i=2}^\xi c_i \) and let \( P_1 \) and \( P_2 \) be the two pairs of pants in \( R \) sharing \( c_1 \) as a common boundary component.
Let \( \alpha_1 \) and \( \alpha_2 \) be the orthogeodesics in \( P_1 \) and \( P_2 \), respectively, connecting \( c_1 \) to itself. 
Up to Dehn twisting about \( c_1 \), there exists a curve homotopic to \( c_1' \) obtained by the taking the union of \( \alpha_1, \alpha_2 \), and two subarcs of \( c_1 \).  It follows that
\[
M_\sigma(w) \leq M + \rm{arccosh}\left(\frac{\cosh(M/2)}{\sinh(m/2)} \right).
\]
The result now follows by the fact that \( \mathcal{MP}(\Sigma) \) is connected with finite diameter, which only depends on \( \xi \).
\end{proof}

Let \( S \) be a ladder surface and 
fix a pants decomposition \( \mathcal P = \{a_k, b_k, c_k \}_{k \in \mathbb Z} \) as in Figure \ref{fig:pants-decomp2}.

\begin{lem}
\label{lem:special-shiga}
If \( X \) is a QCH ladder surface, then there exists a homeomorphism \( f \co S \to X \) such that \( f(\mathcal P) \) is a Shiga pants decomposition for \( X \). 
\end{lem}

\begin{proof}
Let \( \{\gamma_n\}_{n\in \mathbb Z} \) be the collection of curves guaranteed by Lemma~\ref{lem:special curves}.
By Lemma~\ref{lem:area}, the complexity of the surface bounded by \( \gamma_n \) and \( \gamma_{n+1} \), denoted \( Y_n \), is bounded.
By Proposition~\ref{prop:Shiga}, this guarantees the existence of a Shiga pants decomposition for \( X \) containing the collection of curves \( \{ \gamma_n\} \); let \( M \) be an upper bound for the lengths of the cuffs in this decomposition.
Before continuing, we recall that there is a lower bound on the injectivity radius of any \( K \)-QCH surface, which only depends on \( K \) \cite[Theorem 1.1]{BonfertQuasiconformal}; so, let \( K > 1 \) such that \( X \) is \( K \)-QCH and let \( m = m(K) \) denote this lower bound.

Fix a homeomorphism \( f \co S \to X \) such that for every \( n \in \mathbb Z \) there exists \( k_n \in \mathbb Z \) with \( f(c_{k_n}) = \gamma_n \).
Then, \( \mathcal P_n = f(\mathcal P) \cap Y_n \) is a pants decomposition of \( Y_n \).
As the \( Y_n \) have bounded area and hence bounded topological complexity, 
\[
\xi = \max\{ \xi(Y_n) \co n \in \mathbb Z \},
\]
is finite.
Therefore, by Lemma \ref{lem:short-pants}, \( M_{Y_n}([\mathcal P_n]) \) is bounded above by a function of \( \xi, m, \) and \( M \).
In particular, by pre-composing \( f \) with a homeomorphism of \( S \), we may assume that \( f(\mathcal P) \) is a Shiga pants decomposition for \( X \).
\end{proof}

\subsection{Proof of Theorem \ref{thm:qch}}

We can now prove every QCH ladder surface is quasiconformally equivalent to a regular cover a closed surface, finishing the proof of Theorem \ref{thm:qch}.
Before giving the proof, we recall a definition from Teichm\"uller theory.

Let \( S \) be a (topological) ladder surface and let \( \mathcal P = \{ a_k,b_k,c_k\}_{k\in\mathbb Z} \) be the pants decomposition of \( S \) given in Figure~\ref{fig:pants-decomp2}.
Given a hyperbolic surface \( Z \) and a homeomorphism \( h \co S \to Z \), define the Fenchel--Nielsen coordinates of the marked surface \( (S,h) \) be  the collection of sextuplets
\[
\mathrm{FN}((S,h)) = \{[\ell_Z(h(a_k)), \theta_{a_k}(h), \ell_Z(h(b_k)), \theta_{b_k}(h), \ell_Z(h(c_k)), \theta_{c_k}(h)] \}_{k\in \mathbb Z},
\]
where \( \theta_{a_k}, \theta_{b_k}, \) and \( \theta_{c_k} \) are the twist parameters associated to the curves \( a_k, b_k, \) and \( c_k \), respectively, with respect to a collection of seams \( \{d_k\}_{k\in \mathbb Z} \) as in Figure \ref{fig:seams}.
The twist parameters are given as an angle (as opposed to an arc length). 

\begin{figure}
\includegraphics[scale=0.9]{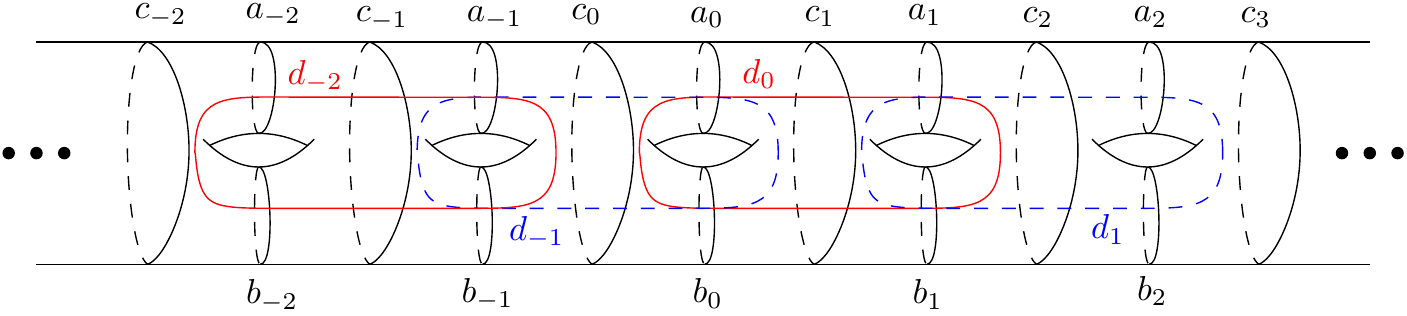}
\caption{The pants decomposition \( \mathcal P \) along with the seams \( \{d_k\}_{k\in \mathbb Z} \) determine Fenchel-Nielsen coordinates for hyperbolic structures on \( S \). }
\label{fig:seams}
\end{figure}

\begin{proof}[Proof of Theorem \ref{thm:qch}]
Let \( X \) be a QCH ladder surface and let \( f \co S \to X \) be the homeomorphism given by Lemma \ref{lem:special-shiga}, so that \( f(\mathcal P) \) is a Shiga pants decomposition for \( X \).
By precomposing \( f \) with a (possibly infinite) product of Dehn twists about the cuffs of \( \mathcal P \), we may assume that the twist parameters for \( X = (S, f) \) are between 0 and \( 2 \pi \).

Now fix the marked surface \( (S, h \co S \to Z) \) such that
\[
\mathrm{FN}((S,h)) = \{ [1,0,1,0,1,0] \}_{k\in \mathbb Z}.
\]
Since the cuff lengths and twist parameters of \( f(\mathcal P) \) are bounded from above and below, we can conclude that the map \( h \circ f^{-1}\co X \to Z \) is quasiconformal \cite[Theorem~8.10]{AlessandriniFenchel}. 

To finish, we show that \( Z \) is a regular cover of a closed  genus-3 hyperbolic surface:
let \( \tau\co S \to S \) be the horizontal translation determined, up to isotopy, by requiring that
\[
\tau((a_k, b_k, c_k, d_k)) = (a_{k+2},b_{k+2}, c_{k+2}, d_{k+2})
\]
for all \( k \in \mathbb Z \).
Observe that 
\[
\mathrm{FN}((S, h\circ \tau^{-1})) = \mathrm{FN}((S,h))
\]
and hence \(\tau^h =  h \circ \tau \circ h^{-1} \co Z \to Z \) is isotopic to an isometry of \( Z \).
It follows that \( \langle \tau^h \rangle \backslash Z \) is a closed hyperbolic genus-3 surface. 
\end{proof}


\section{Topology of QCH surfaces}

The goal of this section is to prove that there are no topologically exotic QCH surfaces, that is:

\begin{thm}
\label{thm:regular}
Every QCH  surface is homeomorphic to a regular cover of a closed surface.
\end{thm}

Before proving Theorem \ref{thm:regular}, we need to understand the topology of a regular cover of a closed surface.
The proposition below is stronger than we require, but with little extra work we state a more complete picture. 
Also, recall that, for the purpose of this article, surfaces are orientable and second countable.

\begin{prop}
\label{prop:cover classification}
A regular cover of a closed surface is either compact or homeomorphic to one of the following six surfaces:

\begin{enumerate}

\item \( \br^2 \),

\item \( \br^2 \ssm  \mathbf 0\),

\item the Cantor tree surface, i.e.~the planar surface whose space of ends is a Cantor space,

\item the blooming Cantor tree surface,  i.e.~the infinite-genus surface with no planar ends and whose space of ends is a Cantor space,

\item the Loch Ness Monster surface, i.e.~the one-ended infinite-genus surface, or

\item the ladder surface, i.e.~the two-ended infinite-genus surface with no planar ends.
 
\end{enumerate}
Moreover, the torus is the only closed surface regularly covered by \( \br^2 \ssm \mathbf 0 \).
\end{prop}

\begin{proof}
Let \( B \) be a closed surface and let \( \pi \co S \to B \) be a regular cover.
It is not difficult to show that the end space of a regular cover of a closed manifold is either empty, discrete with 1 or 2 points, or a Cantor space (this is a classical theorem of Hopf \cite{HopfEnden}).
If the end space of \( S \) is empty, then \( S \) is compact; in the other cases, \( S \) is non-compact.
If \( S \) is non-planar, then the co-compactness of the action of the deck group associated to \( \pi \) on \( S \) will guarantee that every end of \( S \) is non-planar. 
Therefore, either \( S \) is compact; \( S \) is planar with 1, 2, or a Cantor space of ends; or \( S \) is infinite genus with 1, 2, or a Cantor space of ends, all of which are non-planar.
Using the classification of surfaces, we see that---up to homeomorphism---there are only six non-compact surfaces that meet these criteria, namely the ones listed above.
We leave it as an exercise to show that, with the exception of \( \br^2 \ssm \mathbf 0 \), each of the listed surfaces covers a closed genus-2 surface.

Finally,  \( \br^2 \ssm \mathbf 0 \) is a regular cover of the torus; moreover, no surface of genus at least two can be regularly covered by \( \br^2 \ssm \mathbf 0 \): indeed, \( \pi_1( \br^2) \ssm \mathbf 0 \) is cyclic and the fundamental group of a hyperbolic surface cannot have a normal cyclic subgroup. 
\end{proof}

We can now prove Theorem \ref{thm:regular}.

\begin{proof}[Proof of Theorem \ref{thm:regular}]
Let \( X \) be a \( K \)-QCH surface.
Further, for the sake of arguing by contradiction, assume that \( X \) is not a regular cover of a closed surface. 
If \( X \) is closed, then it is trivially a regular cover of a closed surface, namely itself.
So, we may assume that \( X \) is non-compact.
Note that under these assumptions, \( X \) is necessarily hyperbolic. 

First, assume that \( X \) has positive (possibly infinite) genus and has at least one planar end.
We can then choose a non-planar compact subsurface \( Y \) of \( X \) and  an unbounded planar subsurface \( U \) of \( X \) such that \( \partial U \) is compact. 
Let \( \{x_n\}_{n\in\bn} \) be a  sequence in \( U \) such that every compact subset of \( X \) contains only finitely many of the \( x_n \). 
Fix \( x \in Y \) and let \( f_n \co X \to X \) be a \( K \)-quasiconformal map such that \( f_n(x) = x_n \). 
Note that as \( f_n \) is a \( (K, K\log4) \)-quasi-isometry (see Lemma \ref{lem:basics}), the diameter of \( f_n(Y) \) is bounded as a function of \( K \) and the diameter of \( Y \). 
In particular, as \( \rho(x_n, \partial Y) \to \infty \) as \( n \to \infty \), it must be that \( f_n(Y) \subset U \) for large \( n \), but this is impossible as every subsurface of a planar surface is planar.

We can now conclude that \( X \) is either (i) planar or (ii) has infinite genus and no planar ends. 
In either case, as we are assuming \( X \) is not a regular cover of a closed surface, by Proposition \ref{prop:cover classification}, \( X \) has at least three ends, one of which is isolated, call it \( e \).
(Note: if the end space does not contain an isolated point, then it is necessarily a Cantor space.)
Let \( P \) be a compact subsurface in \( X \) with three boundary components, each of which is separating, and such that each component of \( X \ssm P \) is unbounded and such that there exists a component \( U \) of \( X\ssm P \) with \( U \)  a neighborhood of \( e \). 
The argument now proceeds nearly identically to the previous case.
Let \( \{x_n\}_{n\in\bn} \) be a  sequence in \( U \) such that every compact subset of \( X \) contains only finitely many of the \( x_n \). 
Fix \( x \in P \) and let \( f_n \co X \to X \) be a \( K \)-quasiconformal map such that \( f_n(x) = x_n \). 
Again, as \( f_n \) is a \( (K, K\log4) \)-quasi-isometry, the diameter of \( f_n(P) \) is bounded as a function of \( K \) and the diameter of \( P \). 
In particular,  \( \rho(x_n, P) \to \infty \) as \( n \to \infty \), it must be that \( f_n(P) \subset U \) for large \( n \), but this is impossible as it would require \( X \ssm f_n(P) \) to have a bounded component. 
\end{proof}

\bibliographystyle{plain}
\bibliography{references}

\end{document}